\documentclass{amsart}					
\usepackage{amsthm}
\usepackage{amsmath}
\usepackage{amsfonts}
\usepackage{amssymb}
\usepackage{ascmac}
\usepackage{latexsym}
\usepackage[pdftex]{graphicx} 
\usepackage{enumerate}

\theoremstyle{definition}
\newtheorem{Def}{Definition}[section]
\newtheorem{rem}[Def]{Remark}

\newtheorem{ex}[Def]{Example}

\theoremstyle{plain}
\newtheorem{prop}[Def]{Proposition}
\newtheorem{thm}[Def]{Theorem}

\newcommand{\LLS}{{\mathfrak S}}

\def\i<#1>{\langle #1 \rangle}
\def\l<#1>{\left\langle #1 \right\rangle}

\newcommand{\calI}{\mathcal{I}}
\newcommand{\calC}{\mathcal{C}}

\newcommand{\sgn}{\mathrm{sgn}}

\makeatletter
  
  \@addtoreset{equation}{section}
\makeatother

\allowdisplaybreaks[4]
\makeatletter
\def\keywords{\xdef\@thefnmark{}\@footnotetext}
\makeatother

\begin{document}

\title[]{On a generalization of Monge--Amp\`ere equations and Monge--Amp\`ere systems}

\author[Masahiro Kawamata]{Masahiro Kawamata}
\author[Kazuhiro Shibuya]{Kazuhiro Shibuya}


\address[Masahiro Kawamata]{Graduate School of Science, Hiroshima University, 1-3-1 Kagamiyama, Higashi-Hiroshima, 739-8526, Japan}
\address[Kazuhiro Shibuya]{Graduate School of Advanced Science and Engineering, Hiroshima University, 1-3-1 Kagamiyama, Higashi-Hiroshima, 739-8526, Japan}

\keywords{2010 \emph{Mathematics Subject Classification.} Primary 58A15; Secondary 58A17.}%
\keywords{\emph{Key words and phrases.}Monge--Amp\`ere equation; Monge--Amp\`ere system; exterior differential system; jet space}%
    
\maketitle


\begin{abstract}
In the present paper we discuss Monge-Amp\`ere equations from the view point of differential geometry.
It is known that a Monge--Amp\`ere equation corresponds to a special exterior differential system on a 1-jet space. In this paper, we generalize Monge--Amp\`ere equations and prove that a $(k+1)$st order generalized Monge--Amp\`ere equation corresponds to a special exterior differential system on a $k$-jet space and that its solution naturally corresponds to an integral manifold of the corresponding exterior differential system. Moreover, we show that the Korteweg-de Vries (KdV) equation and the Cauchy--Riemann equations are examples of our equations.

\noindent

\end{abstract}

\section{Introduction}
Monge--Amp\`ere equations are given by the following second order partial differential equations:
\begin{flalign*}
		\begin{array}{c}
		Az_{xx}+2Bz_{xy}+Cz_{yy}+D+E(z_{xx}z_{yy}-z_{xy}^{2})=0,
		\end{array}
	\end{flalign*}
where $z=z(x,y)$ is an unknown function and $A,B,C,D$ and $E$ are given functions of $x,y,z,z_{x}$ and $z_{y}$.
These equations contain many important examples (e.g. the wave equation, the heat equation and Laplace's equation), and they have been studied not only by analytical, but also by geometrical methods (see for instance the work of V. V. Lychagin, V. N. Rubtsov, I. V. Chekalov and T. Morimoto which have been used the geometrical theory of differential equations (\cite{LRC,Morimoto95})). On the other hand, various generalizations of Monge--Amp\`ere equations are known. For example, G. Boillat defines a generalization of Monge--Amp\`ere equations from the view point of symplectic geometry. These equations are called the symplectic Monge--Amp\`ere equations which are partial differential equations with $n$ independent variables (\cite{Boillat}).
Another example is the most general nonlinear third order equation which is completely exceptional, so called the third order Monge--Amp\`ere equations obtained by \cite{donate-valenti}.

A geometrical theory of differential equations was developed by \'E. Cartan, \'E. Goursat, etc., to deal with construction of solutions, symmetry of differential equations or classification of differential equations. 
In this theory, differential equations are studied by using the following correspondence: 
a $k$th order differential equation corresponds to a manifold in a $k$-jet space, and its solution corresponds to an integral manifold of an exterior differential system associated to the $k$-jet space (called the canonical system on the $k$-jet space). 
This theory applies to various differential equations. For example, Monge--Amp\`ere equations, first order partial differential equations with one unknown function and Cartan's overdetermined system have been studied by using this theory (\cite{LRC,Morimoto95,Cartan,Hermann}).

According to the geometrical theory of differential equations, a Monge--Amp\`ere equation (which is a second order partial differential equation) corresponds to a manifold in a 2-jet space, and its solution corresponds to an integral manifold of the canonical system on the 2-jet space. On the other hand, T. Morimoto defined a special exterior differential system on a 1-jet space (called Monge--Amp\`ere systems) for each Monge--Amp\`ere equation (\cite{Morimoto95}). Moreover, Morimoto proved that a solution of a Monge--Amp\`ere equation naturally corresponds to an integral manifold of a Monge--Amp\`ere system associated with the equation. 
In other words, the set of all Monge--Amp\`ere equations is a class of partial differential equations which can be studied by means of an exterior differential system on a 1-jet space.

Recently, Monge--Amp\`ere equations have been studied by means of Monge--Amp\`ere systems (\cite{bryant,bryant2,Ishikawa-Morimoto}).

In the present paper, we generalize the symplectic Monge--Amp\`ere equation and the third order Monge--Amp\`ere equation, and call such equations generalized Monge--Amp\`ere equations (GMAE). Moreover, we show that the Korteweg-de Vries (KdV) equation and the Cauchy-Riemann equations are examples of GMAE. We also generalize the Monge--Amp\`ere systems and have the following correspondence GMAE with the generalized Monge--Amp\`ere system (GMAS):

\begin{thm}\label{thm:corresponding}
	For an arbitrary GMAS on a $k$-jet space $J^{k}(n,m)$, there exists a $(k+1)$st order GMAE such that an integral manifold of the GMAS corresponds to a solution of the GMAE. Namely, the following correspondence holds:
	\begin{flalign}\label{intmfd-solution}
		\{\mbox{integral manifolds of the GMAS}\} \leftrightarrow \{\mbox{solutions of the GMAE}\}.
	\end{flalign}
	Conversely, for an arbitrary $(k+1)$st order GMAE, there exists a GMAS on $J^{k}(n,m)$ such that the correspondence (\ref{intmfd-solution}) holds.
\end{thm}
		By this theorem, the set of all $(k+1)$st order GMAEs is a class of partial differential equations which can be studied by means of an exterior differential system on a $k$-jet space.

Our paper is organized as follows. In Section 2, we recall the basics of exterior differential system and jet spaces. In Section 3, we give a generalization of Monge--Amp\`ere equations, Monge--Amp\`ere systems and a relationship between generalized Monge--Amp\`ere equations and generalized Monge--Amp\`ere systems. In Section 4, we describe some examples of generalized Monge--Amp\`ere equations and generalized Monge--Amp\`ere systems.

\section{Preliminaries}

In this section, we recall the foundations of exterior differential systems.  Hereafter, we assume all objects are of class $C^{\infty}$.
Let $\Omega^{\ast}(M)$ and $\Omega^{i}(M)$ be the set of all differential forms and all differential $i$-forms, respectively, on a manifold $M$. For the details of this section, see \cite{BCG3,Cfb}.

\begin{Def}
	A subset $\calI$ of $\Omega^{\ast}(M)$ is called \textit{an exterior differential system} (\textit{EDS} for short) on a manifold $M$, if $\calI$ satisfies the following conditions: 
	\begin{enumerate}
  		\item The subset $\calI$ is a homogeneous ideal of $\Omega^{\ast}(M)$, that is, $\calI$ is written by the following form:
		\[
			\calI = \bigoplus^{\infty}_{i=0}(\calI \cap \Omega^{i}(M)).
		\]
  		\item The subset $\calI$ is closed under exterior differentiation on $M$.
 	\end{enumerate}
\end{Def}

	Let $\{\omega_{1},\ldots,\omega_{r}\}_{\mathrm{diff}}$ denote the ideal generated by homogeneous differential forms 
$\omega_{1},\ldots,\omega_{r}$ $\in \Omega^{\ast}(M)$
\ and its exterior derivatives $d\omega_{1},\ldots,d\omega_{r}$. The ideal $\{\omega_{1},\ldots,$ $\omega_{r}\}_{\mathrm{diff}}$ is an EDS on $M$, and called \textit{the EDS generated by $\omega_{1},\ldots,\omega_{r}$}.

We next recall the jet spaces.  Remark that Darboux's theorem for contact manifolds also holds good for jet spaces (\cite{Yamaguchi82,Yamaguchi83}), in other words, jet spaces have local normal form.
Therefore, in this paper, we define jet spaces locally as follows.

Let $M_{k}$ be the set of all $k$-tuples of integers $1,\ldots,n$. We define an equivalence relation $\sim$ on $M_{k}$ by saying that 
$(i_{1},\ldots,i_{k}) \sim (j_{1},\ldots,j_{k})$, if there exists $\sigma \in \LLS_{k}$ such that $(i_{\sigma(1)},\ldots,i_{\sigma(k)}) = (j_{1},\ldots,j_{k})$, where $\mathfrak{S}_{k}$ is the symmetric group of degree $k$. Now we denote $S_{k} := M_{k}/{\sim}$ and put $\Sigma_{k}:= \bigsqcup^{k}_{j=1} S_{j}$.
Note that $\# S_{k}={}_{n}\mathrm{H}_{k}$, $\# \Sigma_{k}=\sum^{k}_{j=1} {}_{n}\mathrm{H}_{j}$.
Here ${}_{n}\mathrm{H}_{j}$ is the number of combinations with repetition used in \cite{Yamaguchi82,Yamaguchi83}.
For $I=(i_{1},\ldots,i_{k}) \in S_{k}$, by $Ii$, $p_{I}$, and $p_{Ii}$ we denote $(i_{1},\ldots,i_{k},i) \in S_{k+1}$, $p_{i_{1}\cdots i_{k}}$ and $p_{i_{1}\cdots i_{k}i}$, respectively.   

\begin{rem}
	We assume that the symbols $I$ and $Ii$ are elements of $S_{k}$ or $\Sigma_{k}$, not elements of $M_{k}$.
\end{rem}

\begin{Def}
	Let $M$ be a manifold of dimension $(n+m+m\sum^{k}_{j=1}{}_{n}\mathrm{H}_{j})$ and $\calI$ an EDS on $M$. The pair $(M,\calI)$ is called \textit{a $k$-jet space of type (n,m)}, if for any $p \in M$, there exists a local coordinate system $(U; x_{i},z^{\alpha},p^{\alpha}_{I})\ (1 \leq i \leq n,\  1 \leq \alpha \leq m,\ I \in \Sigma_{k})$ around $p$ such that $\calI$ is written by the following form on $U$:
\[
	\calI=\{\omega_{0}^{\alpha},\omega^{\alpha}_{I}\ |1 \leq \alpha \leq m,\ I \in \Sigma_{k-1}\}_{\mathrm{diff}},
\]
where
\[
	\omega^{\alpha}_{0} := dz^{\alpha} - \sum^{n}_{i=1}p^{\alpha}_{i}dx_{i},\ \omega^{\alpha}_{I} := dp^{\alpha}_{I} - \sum^{n}_{i=1}p^{\alpha}_{Ii}dx_{i}
	\hspace{2mm} (1 \leq \alpha \leq m,\ I \in \Sigma_{k-1}).
\]
We write $\calC^{k}:=\calI$ and $(J^{k}(n,m),\calC^{k}) := (M,\calI)$. 
The EDS $\calC^{k}$ is called \textit{the canonical system on $J^{k}(n,m)$} and $(U; x_{i},z^{\alpha},p^{\alpha}_{I})$ are called \textit{a canonical coordinate system around $p$}.
\end{Def}
	
\begin{rem}
	If $k=0$, then a $0$-jet space of type $(n,m)$ is simply a manifold of dimension $(n+m)$. 
	Namely, we define the canonical system $\calC^{0}$ on $J^{0}(n,m)$to be the trivial EDS $\{0\}$.
\end{rem}

We next define contact transformations on a jet space and an integral manifold of an EDS.

	\begin{Def}
		Let $\varphi \colon J^{k}(n,m) \to J^{k}(n,m)$ be a diffeomorphism. Then $\varphi$ is called \textit{a contact transformation (of the $k$th order)} if $\varphi$ satisfies $\varphi^{\ast}(\calC^{k}) = \calC^{k}$.
	\end{Def}

\begin{Def}
	Let $M$ be a manifold and $\calI$ be an EDS on $M$. 
	Then a submanifold $\iota \colon N \hookrightarrow M$ is called \textit{an integral manifold} of $\calI$ if $\iota^{\ast}\omega =0$ for any $\omega \in \calI$. 
\end{Def}

\begin{rem}\label{rem:canonical-system}
	Let $(x_{i},z^{\alpha},p^{\alpha}_{I})\ (1 \leq i \leq n,1 \leq \alpha \leq m,I \in \Sigma_{k})$ be a canonical coordinate system of $J^{k}(n,m)$. Then a submanifold 
	$S=\{(x_{i},z^{\alpha}(x_{1},\ldots,x_{n}),p^{\alpha}_{I}(x_{1}$ $,\ldots,x_{n}))\ |\ 1 \leq i \leq n,1 \leq \alpha \leq m,I \in \Sigma_{k}\}$ in $J^{k}(n,m)$
	\ is an integral manifold of $\calC^{k}$ if and only if, for any $l \in \{1,\ldots,k\}$, $\alpha \in \{1, \ldots ,m\}$ and $I=(i_{1},\ldots,i_{l}) \in S_{l}$, the following holds:
	\[
		p^{\alpha}_{I}(x_{1},\ldots,x_{n})=z^{\alpha}_{x_{i_{1}}\cdots x_{i_{l}}}(x_{1},\ldots,x_{n}).
	\]
\end{rem}

\section{A generalization of Monge--Amp\`ere systems and Monge--Amp\`ere equations}

In this section, we generalize Monge--Amp\`ere systems and Monge--Amp\`ere equations, and give the proof of Theorem~\ref{thm:corresponding}.

\subsection{Generalized Monge--Amp\`ere system}

First of all, we review the Monge--Amp\`ere system defined by T. Morimoto.

\begin{Def}[\cite{Morimoto95}]
	For a jet space $(J^{1}(n,1),\calC^{1})$, let $\omega$ be a generator 1-form of $\calC^{1}$, that is, $\calC^{1}=\{\omega\}_{\mathrm{diff}}$. An EDS $\calI$ on $J^{1}(n,1)$ is called \textit{a Monge--Amp\`ere system} (\textit{MAS} for short), if there exists $\Psi \in \Omega^{n}(J^{1}(n,1))$ such that $\calI$ is generated by $\omega$ and $\Psi$.
\end{Def}

\begin{rem}\label{not-vanished-MAS}
	If $\Psi \equiv 0 \mod \calC^{1}$, then we have $\calI=\{ \omega,\Psi \}_{\mathrm{diff}} = \{\omega\}_{\mathrm{diff}} = \calC^{1}$. In general, we assume $\Psi \not \equiv 0 \mod \calC^{1}$.
\end{rem}

We next define a generalized Monge--Amp\`ere system.

\begin{Def}
	Let $\calC^{k}=\{\omega^{\alpha}_{0},\omega^{\alpha}_{I}\ |\ 1 \leq \alpha \leq m,\  I \in \Sigma_{k-1}\}_{\mathrm{diff}}$ be the canonical system on $J^{k}(n,m)$. 
	An EDS $\calI$ on $J^{k}(n,m)$ is called \textit{a generalized Monge--Amp\`ere system} (\textit{GMAS} for short), if there exist homogeneous differential forms $\Psi_{1},\ldots,\Psi_{r}$ on $J^{k}(n,m)$ such that $\calI$ is generated  by $\omega^{\alpha}_{0},\omega^{\alpha}_{I}$ and $\Psi_{\mu}\ (1 \leq \mu \leq r)$. 
	\end{Def}

Hereafter, we denote $\{\calC^{k},\Psi_{\mu}\ |\ 1 \leq \mu \leq r\}_{\mathrm{diff}}:=\{\omega^{\alpha}_{0},\omega^{\alpha}_{I}, \Psi_{\mu}\ |\ 1 \leq \alpha \leq m,\ I \in \Sigma_{k-1},\ 1 \leq \mu \leq r\}_{\mathrm{diff}}$.

\begin{rem}\label{rem:expression-Psi}
	Let $(x_i,z^\alpha,p^\alpha_{I})\ (1 \leq i \leq n,1 \leq \alpha \leq m, I \in \Sigma_{k})$ be a canonical coordinate system of $J^{k}(n,m)$. Since $\{dx_{i},\omega^{\alpha}_{0},\omega^{\alpha}_{J}, dp^{\alpha}_{I}\ |\ 1 \leq i \leq n,\ 1 \leq \alpha \leq m,\ J \in \Sigma_{k-1},\ I \in S_{k}\}$ is a basis of $\Omega^{1}(J^{k}(n,m))$, the generators $\Psi_{\mu} \in \Omega^{l_{\mu}}(J^{k}(n,m))\ (1 \leq \mu \leq r)$ of the GMAS are expressed as follows:
		\begin{flalign*}
		\Psi_{\mu} &=\mbox{($\omega$'s terms)} +\sum^m_{t=0} \sum^{l}_{\lambda_{0}=0} \sum_{\lambda_1+ \cdots + \lambda_t=l_{\mu}-\lambda_{0}} \sum_{\alpha_1 < \cdots < \alpha_t}\sum_{i_1 < \cdots < i_{\lambda_0}}  \sum_{\substack{I^{\alpha_1}_1 < \cdots < I^{\alpha_1}_{\lambda_1} \\ \cdots \\ I^{\alpha_t}_1 < \cdots < I^{\alpha_t}_{\lambda_t} \\ I^{\alpha_{i}}_{j_{i}} \in S_{k}}}\\
		&\hspace{10mm}A^{I^{\alpha_1}_1 I^{\alpha_1}_2 \ldots I^{\alpha_t}_{\lambda_t}}_{i_1 \ldots i_{\lambda_0}} dx_{i_1} \wedge \cdots \wedge dx_{i_{\lambda_0}} \wedge dp^{\alpha_1}_{I^{\alpha_1}_1} \wedge dp^{\alpha_1}_{I^{\alpha_1}_2} \wedge \cdots \wedge dp^{\alpha_t}_{I^{\alpha_t}_{\lambda_t}}\\
		&\equiv \sum^m_{t=0} \sum^{l}_{\lambda_{0}=0} \sum_{\lambda_1+ \cdots + \lambda_t=l_{\mu}-\lambda_{0}} \sum_{\alpha_1 < \cdots < \alpha_t}\sum_{i_1 < \cdots < i_{\lambda_0}}  \sum_{\substack{I^{\alpha_1}_1 < \cdots < I^{\alpha_1}_{\lambda_1} \\ \cdots \\ I^{\alpha_t}_1 < \cdots < I^{\alpha_t}_{\lambda_t} \\ I^{\alpha_{i}}_{j_{i}} \in S_{k}}}\\
		&\hspace{10mm}A^{I^{\alpha_1}_1 I^{\alpha_1}_2 \ldots I^{\alpha_t}_{\lambda_t}}_{i_1 \ldots i_{\lambda_0}} dx_{i_1} \wedge \cdots \wedge dx_{i_{\lambda_0}} \wedge dp^{\alpha_1}_{I^{\alpha_1}_1} \wedge dp^{\alpha_1}_{I^{\alpha_1}_2} \wedge \cdots \wedge dp^{\alpha_t}_{I^{\alpha_t}_{\lambda_t}} \mod \calC^{k}
			\end{flalign*}
	where $A^{I^{\alpha_1}_1 I^{\alpha_1}_2 \ldots I^{\alpha_t}_{\lambda_t}}_{i_1 \ldots i_{\lambda_0}}$ are functions of $x_{i},z^{\alpha}$ and $p^{\alpha}_{I}$ $(1 \leq i \leq n,\ 1 \leq \alpha \leq m,\ I \in \Sigma_{k})$.
\end{rem}

\begin{rem}\label{rem1}
   			If $n=l, k=1,m=1$ and $r=1$, then the GMAS coincides with an MAS. 
\end{rem}

Finally, we define isomorphisms of GMASs.

\begin{Def}
	Let $\calI_{1}$ and $\calI_{2}$ be GMASs on $J^{k}(n,m)$. Then a diffeomorphism $\varphi \colon J^{k}(n,m) \to J^{k}(n,m)$ is called \textit{an isomorphism of GMASs} if $\varphi$ satisfies the following conditions:
	\begin{enumerate}
		\item $\varphi^{\ast}(\calI_{1})=\calI_{2}$,
		\item $\varphi^{\ast}(\calC^{k})=\calC^{k},\ \text{(i.e. it is a contact transformation)}$.
	\end{enumerate}
\end{Def}

\subsection{Generalized Monge--Amp\`ere equation}
First of all, we recall that the classical Monge--Amp\`ere equation is a ``sum of minors'' of the following matrix
\[
		\left(
			\begin{array}{cc}
				z_{xx} & z_{yx} \\
				z_{xy} & z_{yy}
			\end{array}
		\right),
\]
where $z=z(x,y)$ is an unknown function.
In fact, for all minors order less than or equal to 2 of the above matrix
	\begin{flalign}\label{minors}
		1,\ z_{xx},\ z_{xy},\ z_{yy},\ z_{xx}z_{yy}-z_{xy}^{2},
	\end{flalign}
we consider a functional coefficients linear combination of (\ref{minors}). Here we define the minor of order 0 by 1. Then we have the classical Monge--Amp\`ere equation
\begin{flalign}\label{MAE}
		\begin{array}{c}
		Az_{xx}+2Bz_{xy}+Cz_{yy}+D+E(z_{xx}z_{yy}-z_{xy}^{2})=0, \\
		\end{array}
\end{flalign}
where $A, B, C, D,$ and $E$ are given functions of $x, y, z, z_{x}$ and $z_{y}$. Generalized Monge-Amp\`ere equations are also defined by partial differential equations whose forms are ``sum of minors''.

Before giving a definition of generalized Monge--Amp\`ere equations, we remark that $S_{k}$ (see Section~2) has the following order. For $I=(i_{1},\ldots,i_{k}), J=(j_{1},\ldots,j_{k})$ $\in S_{k}$, then $I \leq J$ if and only if either $i_{s}=j_{s}$ for all $s \in \{1,\ldots,n\}$ or there exists 
$s \in \{1,\ldots,n\}$ such that $i_{s}<j_{s}$ and $i_{t}=j_{t}$ for all $t<s$. This order is called the lexicographic order. Hereafter, we assume that $S_{k}$ is ordered in this fashion.

On the other hand, for $I=(i_{1},\ldots,i_{k}) \in S_{k}$, let $z^{\alpha}_{I}$ and $z^{\alpha}_{Ii}$ mean $z^{\alpha}_{x_{i_{1}}\ldots x_{i_{k}}}$ and $z^{\alpha}_{Ij}=z^{\alpha}_{x_{i_{1}}\ldots x_{i_{k}}x_{j}}$. Now we define the $n \times (m \cdot {}_{n}\mathrm{H}_{k})$ matrix
\[
	M(n,m;k):=\left(
			\begin{array}{ccc}
				(z^1_{I1})_{I \in S_k} & \cdots & (z^m_{I1})_{I \in S_k}\\
				\vdots &  & \vdots \\
				(z^1_{In})_{I \in S_k} & \cdots & (z^m_{In})_{I \in S_k}\\
			\end{array}
		\right).
\]
Here, $(z^{\alpha}_{Ij})_{I \in S_{k}}$ is a row vector 
whose components $z^{\alpha}_{Ij}$ are arranged from the left to right with respect to the above order of $S_{k}$.
Remark that the components of $M(n,m;k)$ are $(k+1)$st order partial derivatives of the functions $z^{\alpha}$, not vector functions $(z^{\alpha}_{Ij})_{I \in S_{k}}$.
\begin{ex}
	If $n=2, m=1$ and $k=1$, then we have $S_{1}=\{1,2\}$ and 
	$(z_{I1})_{I \in S_{1}}=(z_{x_{1}x_{1}},z_{x_{2}x_{1}}),\ (z_{I2})_{I \in S_{1}}=(z_{x_{1}x_{2}},z_{x_{2}x_{2}}).$
	Hence  
	\begin{flalign}\label{matrix1}
		M(2,1;1)=\left(
			\begin{array}{cc}
				z_{x_{1}x_{1}} & z_{x_{2}x_{1}} \\
				z_{x_{1}x_{2}} & z_{x_{2}x_{2}}\\
			\end{array}
		\right).
	\end{flalign}
\end{ex}

Now we define a generalized Monge--Amp\`ere equation.

\begin{Def}\label{Def:GMAE}
	 For natural numbers $n,m,k$ and $l\ (1 \leq l \leq n)$, a system of partial differential equation obtained by the following procedure is called \textit{a generalized Monge--Amp\`ere equation} (\textit{GMAE} for short): 

\begin{enumerate}[\bfseries {Step} 1]
		\setlength{\parskip}{0mm}
		\setlength{\itemsep}{3mm}
		\item We choose $\nu_{1}\mbox{th row},\ldots,\nu_{l}\mbox{th row}\ (1 \leq \nu_{1}<\cdots<\nu_{l} \leq n)$ from the $n \times (m\cdot {}_{n}\mathrm{H}_{k})$ matrix $M(n,m;k)$.
		\item We consider all minors with order less than or equal to $l$ of the chosen matrix in Step 1. Then these minors are written by the following form: 
	\begin{flalign} \label{minor}
			\begin{array}{l}\displaystyle
					H^{I^{\alpha_1}_1 \ldots I^{\alpha_1}_{\lambda_1} \ldots I^{\alpha_t}_1 \ldots I^{\alpha_t}_{\lambda_t}}_{j^{\alpha_1}_1 \ldots j^{\alpha_1}_{\lambda_1} \ldots  j^{\alpha_t}_1 \ldots j^{\alpha_t}_{\lambda_t} }(z^{\alpha_{1}},\ldots,z^{\alpha_{t}})\\
					\displaystyle
					:=\left|
					\begin{array}{ccccccc}
						z^{\alpha_1}_{I^{\alpha_1}_1j^{\alpha_1}_1} & \cdots & z^{\alpha_1}_{I^{\alpha_1}_{\lambda_1}j^{\alpha_1}_1} & &z^{\alpha_t}_{I^{\alpha_t}_{1}j^{\alpha_1}_1}& \cdots & z^{\alpha_t}_{I^{\alpha_t}_{\lambda_t}j^{\alpha_1}_1}\\
						\vdots & & \vdots &\cdots & \vdots& & \vdots \\
						z^{\alpha_1}_{I^{\alpha_1}_1j^{\alpha_t}_{\lambda_t}} & \cdots & z^{\alpha_1}_{I^{\alpha_1}_{\lambda_1}j^{\alpha_t}_{\lambda_t}} & &z^{\alpha_t}_{I^{\alpha_t}_{\lambda_t}j^{\alpha_t}_{\lambda_t}} &\cdots & z^{\alpha_t}_{I^{\alpha_t}_{1}j^{\alpha_t}_{\lambda_t}}\\
					\end{array}
					\right|.
				\end{array}
	\end{flalign}
	Here, $t$ is the number of unknown functions that appears in (\ref{minor}), $\lambda_{i}\ (1 \leq i \leq  t)$ are the number of partial derivatives of $z^{\alpha_{i}}$ that appears in (\ref{minor}) and $\lambda_{0}$ is the number of rows that not selected among $\nu_{1}$th row, $\ldots$, $\nu_{l}$th row. Then $\sum^{t}_{r=1}\lambda_{r}=l-\lambda_{0}$ holds.
		\item We put
		\[
			\{i_{1},\ldots,i_{\lambda_{0}}\} := \{\nu_{1},\ldots,\nu_{l}\} \setminus \{j^{\alpha_{1}}_{1},j^{\alpha_{1}}_{2},\ldots, j^{\alpha_{t}}_{\lambda_{t}}\} \hspace{2mm} (1 \leq i_{1} < \cdots < i_{\lambda_{0}} \leq n),
		\]
		 and consider a linear combination of (\ref{minor}) with functional coefficients
			\[
				\sgn
					\left(
						\begin{array}{c}
							i_1,\ldots, i_{\lambda_0}, j^{\alpha_1}_1, j^{\alpha_1}_2, \ldots , j^{\alpha_t}_{\lambda_{t}} \\
							\nu_1 ,\ldots, \nu_{\lambda_0},\nu_{\lambda_0+1},\nu_{\lambda_0+2} ,\ldots, \nu_{l} 
						\end{array}
					\right)
					A^{I^{\alpha_1}_1 I^{\alpha_1}_2 \ldots I^{\alpha_t}_{\lambda_t}}_{i_1 \ldots i_{\lambda_0}},
			\]
			where $A^{I^{\alpha_1}_1 I^{\alpha_1}_2 \ldots I^{\alpha_t}_{\lambda_t}}_{i_1 \ldots i_{\lambda_0}}$ are functions of $x_{i},z^{\alpha}$ and $p^{\alpha}_{I}\ (1 \leq i \leq n, 1 \leq \alpha \leq m, I \in \Sigma_{k})$. Then we obtain the following formula:
	\begin{flalign}
		\begin{array}{l}\label{GMAE0}
			\displaystyle
			\sum^m_{t=1} \sum^{l}_{\lambda_{0}=0}\sum_{\lambda_1+ \cdots + \lambda_t=l-\lambda_{0}}  \sum_{\alpha_1 < \cdots < \alpha_t} \sum_{\substack{i_1 < \cdots < i_{\lambda_0} \\ j^{\alpha_1}_1 <j^{\alpha_1}_2< \cdots < j^{\alpha_t}_{\lambda_t}  \\  \{ i_1, \ldots , j^{\alpha_t}_{\lambda_t} \} = \{ \nu_1 ,\ldots, \nu_l \}}}\sum_{\substack{I^{\alpha_1}_1 < \cdots < I^{\alpha_1}_{\lambda_1} \\ \cdots \\ I^{\alpha_t}_1 < \cdots < I^{\alpha_t}_{\lambda_t}}} \\
		\hspace{10mm}
		A^{I^{\alpha_1}_1 I^{\alpha_1}_2 \ldots I^{\alpha_t}_{\lambda_t}}_{i_1 \ldots i_{\lambda_0} } 
\Delta^{I^{\alpha_1}_1 \ldots I^{\alpha_1}_{\lambda_1} \ldots I^{\alpha_t}_1 \ldots I^{\alpha_t}_{\lambda_t}}_{j^{\alpha_1}_1 \ldots j^{\alpha_1}_{\lambda_1} \ldots  j^{\alpha_t}_1 \ldots j^{\alpha_t}_{\lambda_t} }(\nu_1,\ldots,\nu_l\ ; z^{\alpha_1}, \ldots, z^{\alpha_t}),
		\end{array}
	\end{flalign}
	where 
	{\footnotesize
	\begin{flalign*}
		&\Delta^{I^{\alpha_1}_1 \ldots I^{\alpha_1}_{\lambda_1} \ldots I^{\alpha_t}_1 \ldots I^{\alpha_t}_{\lambda_t}}_{j^{\alpha_1}_1 \ldots j^{\alpha_1}_{\lambda_1} \ldots  j^{\alpha_t}_1 \ldots j^{\alpha_t}_{\lambda_t} }(\nu_1,\ldots,\nu_l\ ; z^{\alpha_1}, \ldots, z^{\alpha_t})\\
		&:=\sgn
	\left(
		\begin{array}{c}
			i_1,\ldots, i_{\lambda_0}, j^{\alpha_1}_1, j^{\alpha_1}_2, \ldots , j^{\alpha_t}_{\lambda_{t}} \\
			\nu_1 ,\ldots, \nu_{\lambda_0},\nu_{\lambda_0+1},\nu_{\lambda_0+2} ,\ldots, \nu_{l} 
		\end{array}
	\right)
H^{I^{\alpha_1}_1 \ldots I^{\alpha_1}_{\lambda_1} \ldots I^{\alpha_t}_1 \ldots I^{\alpha_t}_{\lambda_t}}_{j^{\alpha_1}_1 \ldots j^{\alpha_1}_{\lambda_1} \ldots  j^{\alpha_t}_1 \ldots j^{\alpha_t}_{\lambda_t} }(z^{\alpha_{1}},\ldots,z^{\alpha_{t}}).
	\end{flalign*}
	}
	We will consider the single partial differential equation given by formula (\ref{GMAE0}) put equal to zero.
	\renewcommand{\arraystretch}{1}
		\item Finally, we repeat the above Step 1 to Step 3 for all cases by choosing rows from $M(n,m;k)$ and gather together these equations. We obtain the following system of partial differential equations:
	\begin{flalign}
		\begin{array}{l}\label{GMAE2-Step4}
			\displaystyle
			\sum^m_{t=1} \sum^{l}_{\lambda_{0}=0}\sum_{\lambda_1+ \cdots + \lambda_t=l-\lambda_{0}}  \sum_{\alpha_1 < \cdots < \alpha_t} \sum_{\substack{i_1 < \cdots < i_{\lambda_0} \\ j^{\alpha_1}_1 <j^{\alpha_1}_2< \cdots < j^{\alpha_t}_{\lambda_t}  \\  \{ i_1, \ldots , j^{\alpha_t}_{\lambda_t} \} = \{ \nu_1 ,\ldots, \nu_l \}}}\sum_{\substack{I^{\alpha_1}_1 < \cdots < I^{\alpha_1}_{\lambda_1} \\ \cdots \\ I^{\alpha_t}_1 < \cdots < I^{\alpha_t}_{\lambda_t}}} \\
		\hspace{2mm}
	A^{I^{\alpha_1}_1 I^{\alpha_1}_2 \ldots I^{\alpha_t}_{\lambda_t}}_{i_1 \ldots i_{\lambda_0} } 
\Delta^{I^{\alpha_1}_1 \ldots I^{\alpha_1}_{\lambda_1} \ldots I^{\alpha_t}_1 \ldots I^{\alpha_t}_{\lambda_t}}_{j^{\alpha_1}_1 \ldots j^{\alpha_1}_{\lambda_1} \ldots  j^{\alpha_t}_1 \ldots j^{\alpha_t}_{\lambda_t} }(\nu_1,\ldots,\nu_l\ ; z^{\alpha_1}, \ldots, z^{\alpha_t})		 =0 \\
		 \hspace{55mm}(1 \leq \nu_{1}<\cdots<\nu_{l} \leq n).
		\end{array}
	\end{flalign}
	\renewcommand{\arraystretch}{1}
	\end{enumerate}
\end{Def}

\begin{rem}
	A GMAE is a system of $(k+1)$st order partial differential equations with $n$ independent variables and $m$ unknown functions. Then the number of equations of the GMAE obtained from $M(n,m;k)$ is the combination $\left(
		\begin{array}{c}
			n \\
			l
		\end{array}
	\right)$.
\end{rem}

\begin{Def}\label{def:GMAE-general}
	Let $n,m$ and $k$ be natural numbers. Then, for a natural number $l\ (1 \leq l \leq n)$, we denote a GMAE in Definition~\ref{Def:GMAE} by $F(n,m,k,l)=0$. Let $r, l_{\mu}\ (1 \leq \mu \leq r,\ 1 \leq l_{\mu} \leq n)$ be natural numbers and $F_{\mu}(n,m,k,l_{\mu})=0\ (1 \leq \mu \leq r)$ be GMAEs. A system of ``systems of partial differential equations'' $F_{1}(n,m,k,l_{1})=0,\ldots,F_{r}(n,m,k,l_{r})=0$ is also called a generalized Monge-Amp\`ere equation.
\end{Def}

\begin{rem}
	For a natural number $n,m,k,r$ and $l_{\mu}\ (1 \leq \mu \leq r,\ 1 \leq l_{\mu} \leq n)$, the number of equations of a GMAE in Definition~\ref{def:GMAE-general} is
		$\sum^{r}_{\mu =1}
		\left(
		\begin{array}{c}
			n \\
			l_{\mu}
		\end{array}
	\right)$ (see Example~\ref{ex:system of GMAE}).
\end{rem}

\begin{ex}\label{ex:overdetermined}
		 We consider a GMAE of the case $n=2,m=1,k=1,r=1$ and $l=1$ from Definition~\ref{def:GMAE-general}. Then we have $2 \times (1 \cdot {}_{2}\mathrm{H}_{1})$ matrix
	\begin{flalign}\label{GMAE-r-1}
		M(2,1;1)=\left(
			\begin{array}{cc}
				z_{x_{1}x_{1}} & z_{x_{2}x_{1}} \\
				z_{x_{1}x_{2}} & z_{x_{2}x_{2}}\\
			\end{array}
		\right).
	\end{flalign}
	\begin{enumerate}[\bfseries {Step} 1]
		\item Now we pick up the first row in $M(2,1;1)$:
		\[
			(z_{x_{1}x_{1}},z_{x_{2}x_{1}}).
		\]
		\item We next consider minors with order less than or equal to 1 in the matrix picked up in Step 1, that is
		\[
			1,\ z_{x_{1}x_{1}},\ z_{x_{2}x_{1}}.
		\]
		\item In this case, since $\lambda_{0}$ equal to $0$ or $1$, then $\{i_{1},\ldots,i_{\lambda_{0}}\}$ is equal to $\varnothing$ or $\{1\}$.
		Hence we have the following single partial differential equation:
		\begin{flalign}\label{ex:eq1}
			A^{1}z_{x_{1}x_{1}}+A^{2}z_{x_{2}x_{1}}+A_{1}=0,
		\end{flalign}
		where $A_{1},A^{1}$ and $A^{2}$ are functions of $x_{1},x_{2},z,p_{1}$ and $p_{2}$.
		\item The same procedure is performed for the second row in $M(2,1;1)$. Then we have
		\begin{flalign}\label{ex:eq2}
			A^{1}z_{x_{1}x_{2}}+A^{2}z_{x_{2}x_{2}}+A_{2}=0,
		\end{flalign}
	where $A_{2}$ is a function of $x_{1},x_{2},z,p_{1}$ and $p_{2}$.
	Finally, we gather together the equations (\ref{ex:eq1}) and  (\ref{ex:eq2}) obtaining the following overdetermined system of partial differential equations:
	\begin{flalign}\label{ex:eq3}
		\begin{array}{c}
		\left\{
			\begin{array}{l}
				A^{1}z_{x_{1}x_{1}}+A^{2}z_{x_{2}x_{1}}+A_{1}=0 \\
				A^{1}z_{x_{1}x_{2}}+A^{2}z_{x_{2}x_{2}}+A_{2}=0
			\end{array}
		\right. .
		\end{array}
	\end{flalign}
	The equation (\ref{ex:eq3}) is the GMAE of the case $n=2,m=1,k=1$ and $l=1$.
	\end{enumerate}
\end{ex}

\begin{ex}\label{ex:system of GMAE}
	We next consider a GMAE of the case $n=2,m=1,k=1,r=2,l_{1}=1$ and $l_{2}=1$ from Definition~\ref{def:GMAE-general}. By arranging of two GMAEs obtained by Example~\ref{ex:overdetermined}, the GMAE in this case is the following form: 
	\begin{flalign*}
		\left\{
			\begin{array}{l}
				A^{1}z_{x_{1}x_{1}}+A^{2}z_{x_{2}x_{1}}+A_{1}=0 \\
				A^{1}z_{x_{1}x_{2}}+A^{2}z_{x_{2}x_{2}}+A_{2}=0\\
				B^{1}z_{x_{1}x_{1}}+B^{2}z_{x_{2}x_{1}}+B_{1}=0 \\
				B^{1}z_{x_{1}x_{2}}+B^{2}z_{x_{2}x_{2}}+B_{2}=0
			\end{array}
		\right.,
	\end{flalign*}
	where $A^{i},A_{i},B^{i}$ and $B_{i}\ (1 \leq i \leq 2)$ are functions of $x_{1},x_{2},z,p_{1}$ and $p_{2}$.
\end{ex}

\subsection{Proof of Theorem~\ref{thm:corresponding}}
In this subsection, we prove Theorem~\ref{thm:corresponding} which is the main theorem in this paper.

Let $\Psi$ be an $l$-form on a manifold $M$ and $S$ be an $n$-dimensional submanifold in $M\ (\iota \colon S \hookrightarrow M)$. If $l > n$, then $\iota^{\ast}\Psi =0$. Hence we have the following proposition.

\begin{prop}\label{prop:not-vanished-GMAS2}
	Let $(J^{k}(n,m),\calC^{k})$ be a $k$-jet space of type $(n,m)$, 
	$(x_i,z^\alpha,p^\alpha_{I})$ $(1 \leq i \leq n,1 \leq \alpha \leq m, I \in \Sigma_{k})$
	 be a canonical coordinate system of $J^{k}(n,m)$ and $S =\{(x_i, z^\alpha (x_1,\ldots,x_n),z^\alpha_I(x_1,\ldots,x_n))\ |\ 1\leq i \leq n,\ 1 \leq \alpha \leq m, I \in \Sigma_{k}\}$ be a submanifold in $J^{k}(n,m)$. For a differential $l_{\mu}$-form $\Psi_{\mu}$ on $J^{k}(n,m)\ (1 \leq \mu \leq r)$, if there exists $\mu_{0} \in \{1, \ldots, r\}$ such that $l_{\mu_{0}} > n$, then $S$ is an integral manifold of the GMAS $\{\calC^{k},\Psi_{\mu}\ |\ 1 \leq \mu \leq r\}_{\mathrm{diff}}$ if and only if $S$ is an integral manifold of the GMAS $\{\calC^{k},\Psi_{\mu}\ |\ 1 \leq \mu \leq  r,\ \mu \ne \mu_{0}\}_{\mathrm{diff}}$.
\end{prop}

By Proposition~\ref{prop:not-vanished-GMAS2}, we assume that $l_{\mu} \leq n$ for any $\mu \in \{1,\ldots,r\}$ in this paper. 

Now we prove Theorem~\ref{thm:corresponding}. The proof is divided into three cases, $m=r=1$, $r=1$ and the general case. First of all, we prove Theorem~\ref{thm:corresponding} in the case $m=r=1$. In this case, Theorem~\ref{thm:corresponding} holds good from the following Proposition~\ref{GMAS-GMAE} and Proposition~\ref{GMAE-GMAS}.

\begin{prop}[the case of $m=r=1$]\label{GMAS-GMAE}
	Let $(x_i,z,p_{I})\ (1 \leq i \leq n, I \in \Sigma_{k})$ be a canonical coordinate system of $J^{k}(n,1)$ and $\ \calI=\{\calC^{k},\Psi\}_{\mathrm{diff}}$ be a GMAS on $J^{k}(n,m)$. 
	By Remark~\ref{rem:expression-Psi}, we express $\Psi \in \Omega^{l}(J^{k}(n,1))\ (1 \leq l \leq n)$ as
	\[
		\Psi \equiv \sum^l_{\lambda=0}\sum_{i_1 < \ldots <i_\lambda} \sum_{\substack{I_1 < \ldots < I_{l-\lambda} \\ I_j \in S_k}}A^{I_1 \ldots I_{l - \lambda}}_{i_1 \ldots i_\lambda } dx_{i_1}\wedge \cdots \wedge dx_{i_\lambda} \wedge dp_{I_1} \wedge \cdots \wedge dp_{I_{l-\lambda}} \mod \calC^{k}.
	\]
	Then the submanifold $S$ of $J^{k}(n,1)$ given by
	\[
		S=\{(x_i, z(x_1,\ldots,x_n),z_I(x_1,\ldots,x_n))\ |\ 1\leq i \leq n,\ I \in \Sigma_{k}\}
	\]
	is an integral manifold of $\ \calI$ if and only if $z(x_1,\ldots,x_n)$ is a solution of the following GMAE:	
	\[
		\begin{array}{l}
		\displaystyle\sum_{\lambda=0}^{l} \sum_{\substack{i_1 < \cdots < i_\lambda \\ j_1 < \cdots <j_{l-\lambda}  \\ \{i_1,\ldots,j_{l-\lambda}\} = \{\nu_1,\ldots, \nu_l\}}} \sum_{\substack{I_1 < \cdots < I_{l-\lambda} \\ I_j \in S_k}} A^{I_1 \ldots I_{l-\lambda}}_{i_1 \ldots i_\lambda} \Delta_{j_1 \ldots j_{l-\lambda}}^{I_1 \ldots I_{l-\lambda}} (\nu_1 ,\ldots, \nu_l\ ; z) =0\\
		\hspace{70mm} (1 \leq \nu_1 < \cdots < \nu_l \leq n).
		\end{array}
	\]
\end{prop}

\begin{proof}
	Let $\iota \colon S \hookrightarrow J^{k}(n,1)$ be an $n$-dimensional integral manifold of $\calI$ such that $x_{1},\ldots,x_{n}$ are coordinates of $S$. Then $S$ is an integral manifold of $\calC^{k}$, and we have $p_{I}(x_{1},\ldots,x_{n})=z_{I}(x_{1},\ldots,x_{n})\ ( I \in \Sigma_{k})$ by Remark~\ref{rem:canonical-system}. Therefore
	\begin{flalign}
		\iota^\ast \Psi 
		&= \sum^l_{\lambda=0}\sum_{i_1 < \ldots <i_\lambda} \sum_{\substack{I_1 < \ldots < I_{l-\lambda} \\ I_j \in S_k}}A^{I_1 \ldots I_{l - \lambda}}_{i_1 \ldots i_\lambda} dx_{i_1}\wedge \cdots dx_{i_\lambda} \wedge dz_{I_1} \wedge \cdots \wedge dz_{I_{l-\lambda}} \notag \\
		& \begin{array}{l}\displaystyle
			= \sum^l_{\lambda=0}\sum_{i_1 < \ldots <i_\lambda} \sum_{\substack{I_1 < \ldots < I_{l-\lambda} \\ I_j \in S_k}}\sum_{j_1, \ldots , j_{l-\lambda}}\\
		\hspace{20mm}A^{I_1 \ldots I_{l - \lambda}}_{i_1 \ldots i_\lambda }z_{I_1j_1}\cdots z_{I_{l-\lambda}j_{l-\lambda}}  dx_{i_1}\wedge \cdots dx_{i_\lambda} \wedge dx_{j_1} \wedge \cdots \wedge dx_{j_{l-\lambda}}
		\end{array}\notag \\
		& \begin{array}{l}
			\displaystyle
			=\sum^l_{\lambda=0}
		\sum_{\substack{{i_1 < \ldots <i_\lambda} \\ j_1<\cdots < j_{l-\lambda}}} \sum_{\substack{I_1 < \ldots < I_{l-\lambda} \\ I_j \in S_k}}\\
		\hspace{20mm}A^{I_1 \ldots I_{l - \lambda}}_{i_1 \ldots i_\lambda }H^{I_{1}\ldots I_{l-\lambda}}_{j_{1}\ldots j_{l-\lambda}}(z)
		dx_{i_1}\wedge \cdots dx_{i_\lambda} \wedge dx_{j_1} \wedge \cdots \wedge dx_{j_{l-\lambda}}
		\end{array} \notag\\
		&\label{calc1}\begin{aligned}
		&=\sum_{\nu_1 < \cdots <\nu_l}\sum^l_{\lambda=0}\sum_{\substack{{i_1 < \ldots <i_\lambda} \\ j_1<\cdots < j_{l-\lambda} \\ \{ i_1 ,\ldots,i_\lambda,j_{1},\ldots,j_{l-\lambda} \} = \{   \nu_1 ,\ldots , \nu_l\}}} \sum_{\substack{I_1 < \ldots < I_{l-\lambda} \\ I_j \in S_k}} \\
		&\hspace{20mm}A^{I_1 \ldots I_{l - \lambda}}_{i_1 \ldots i_\lambda}\Delta_{I_1 \ldots I_{l-\lambda}}^{j_1 \ldots j_{l-\lambda}} (\nu_1 ,\ldots, \nu_l\ ; z) dx_{\nu_1}\wedge \cdots \wedge dx_{\nu_l}.
		\end{aligned}
	\end{flalign}
	Furthermore, $dx_{\nu_{1}}\wedge \cdots \wedge dx_{\nu_{l}}\ (1 \leq \nu_{1} < \cdots < \nu_{l} \leq n)$ are linearly independent on $S$ and hence,  by using $\iota^{\ast}\Psi =0$, we obtain
	\[
		\begin{array}{l}
			\displaystyle
			\sum_{\lambda=0}^{l} \sum_{\substack{i_1 < \cdots < i_\lambda \\ j_1 < \cdots <j_{l-\lambda}  \\ \{i_1,\ldots,j_{l-\lambda}\} = \{\nu_1,\ldots, \nu_l\}}} \sum_{\substack{I_1 < \cdots < I_{l-\lambda} \\ I_j \in S_k}} A^{I_1 \ldots I_{l-\lambda}}_{i_1 \ldots i_\lambda} \Delta_{j_1 \ldots j_{l-\lambda}}^{I_1 \ldots I_{l-\lambda}} (\nu_1 ,\ldots, \nu_l\ ; z) =0\\
			\hspace{70mm} (1 \leq \nu_1 < \cdots < \nu_l \leq n).
		\end{array}
	\]
	The converse is proved by substituting a solution of GMAE for $z$ in (\ref{calc1}).
\end{proof}

\begin{prop}[the case of $m=r=1$]\label{GMAE-GMAS}
	For an arbitrary GMAE in the case $m=r=1$, namely
	\begin{flalign}\label{GMAE-r=m=1}
		\hspace{-2mm}\begin{array}{l}
			\displaystyle
			\sum_{\lambda=0}^{l} \sum_{\substack{i_1 < \cdots < i_\lambda \\ j_1 < \cdots <j_{l-\lambda}  \\ \{i_1,\ldots,j_{l-\lambda}\} = \{\nu_1,\ldots, \nu_l\}}} \sum_{\substack{I_1 < \cdots < I_{l-\lambda} \\ I_j \in S_k}} A^{I_1 \ldots I_{l-\lambda}}_{i_1 \ldots i_\lambda} \Delta_{j_1 \ldots j_{l-\lambda}}^{I_1 \ldots I_{l-\lambda}} (\nu_1 ,\ldots, \nu_l\ ; z) =0\\
			\hspace{70mm}(1 \leq \nu_1 < \cdots < \nu_l \leq n), 
		\end{array}
	\end{flalign}
	a function $z(x_{1},\ldots,x_{n})$ is a solution of (\ref{GMAE-r=m=1}) if and only if the submanifold $S$ of $J^{k}(n,1)$ given by
	\[
		S=\{(x_i, z(x_1,\ldots,x_n),z_I(x_1,\ldots,x_n))\ |\ 1\leq i \leq n,\ I \in \Sigma_{k}\} 
	\]
	is an integral manifold of the GMAS generated by the canonical system $\calC^{k}$ on $J^{k}(n,1)$ and the following $l$-form $\Psi$ on $J^{k}(n,1)$:
	\[
		\Psi \equiv \sum^l_{\lambda=0}\sum_{i_1 < \ldots <i_\lambda} \sum_{\substack{I_1 < \ldots < I_{l-\lambda} \\ I_j \in S_k}}A^{I_1 \ldots I_{l - \lambda}}_{i_1 \ldots i_\lambda } dx_{i_1}\wedge \cdots \wedge dx_{i_\lambda} \wedge dp_{I_1} \wedge \cdots \wedge dp_{I_{l-\lambda}} \mod \calC^{k},
	\]
	where
		$(x_i,z,p_{I})\ (1 \leq i \leq n,\ I \in \Sigma_{k})$ be a canonical coordinate system of $J^{k}(n,1)$.
	\end{prop}

\begin{proof}
	The proof is similar to that of Proposition~\ref{GMAS-GMAE}.
\end{proof}

We next prove Theorem~\ref{thm:corresponding} in the case $r=1$. In this case, Theorem~\ref{thm:corresponding} holds good from the following Proposition~\ref{GMAS-GMAE2} and Proposition~\ref{GMAE-GMAS2}

\begin{prop}[the case of $r=1$]\label{GMAS-GMAE2}
	Let $(x_i,z^{\alpha},p^{\alpha}_{I})\ (1 \leq i \leq n,\ 1 \leq \alpha \leq m,\ I \in \Sigma_{k})$ be a canonical coordinate system of $J^{k}(n,m)$ and $\ \calI=\{\calC^{k},\Psi\}_{\mathrm{diff}}$ be a GMAS on $J^{k}(n,m)$. By Remark~\ref{rem:expression-Psi}, we express $\Psi \in \Omega^{l}(J^{k}(n,m))\ (1 \leq l \leq n)$ as
	\begin{flalign*}
		&\Psi \equiv \sum^m_{t=0} \sum^{l}_{\lambda_{0}=0} \sum_{\lambda_1+ \cdots + \lambda_t=l-\lambda_{0}} \sum_{\alpha_1 < \cdots < \alpha_t}\sum_{i_1 < \cdots < i_{\lambda_0}}  \sum_{\substack{I^{\alpha_1}_1 < \cdots < I^{\alpha_1}_{\lambda_1} \\ \cdots \\ I^{\alpha_t}_1 < \cdots < I^{\alpha_t}_{\lambda_t} \\ I^{\alpha_{i}}_{j_{i}} \in S_{k}}}\\
		&\hspace{15mm}A^{I^{\alpha_1}_1 I^{\alpha_1}_2 \ldots I^{\alpha_t}_{\lambda_t}}_{i_1 \ldots i_{\lambda_0}} dx_{i_1} \wedge \cdots \wedge dx_{i_{\lambda_0}} \wedge dp^{\alpha_1}_{I^{\alpha_1}_1} \wedge dp^{\alpha_1}_{I^{\alpha_1}_2} \wedge \cdots \wedge dp^{\alpha_t}_{I^{\alpha_t}_{\lambda_t}}\mod \calC^{k}.
	\end{flalign*}
	Then the submanifold $S$ of $J^{k}(n,m)$ given by
	\[
		S=\{(x_i, z^{\alpha}(x_1,\ldots,x_n),z^{\alpha}_I(x_1,\ldots,x_n))\ |\ 1\leq i \leq n,\ 1 \leq \alpha \leq m,\ I \in \Sigma_{k}\}
	\]
	is an integral manifold of $\ \calI$ if and only if $z^{\alpha}(x_1,\ldots,x_n)\ (1 \leq \alpha \leq m)$ is a solution of the following GMAE:	
	\renewcommand{\arraystretch}{2.5}
	\begin{flalign}
		\begin{array}{l}\label{GMAS}
			\displaystyle
			\sum^m_{t=1} \sum^{l}_{\lambda_{0}=0}\sum_{\lambda_1+ \cdots + \lambda_t=l-\lambda_{0}}  \sum_{\alpha_1 < \cdots < \alpha_t} \sum_{\substack{i_1 < \cdots < i_{\lambda_0} \\ j^{\alpha_1}_1 <j^{\alpha_1}_2< \cdots < j^{\alpha_t}_{\lambda_t}  \\  \{ i_1, \ldots , j^{\alpha_t}_{\lambda_t} \} = \{ \nu_1 ,\ldots, \nu_l \}}}\sum_{\substack{I^{\alpha_1}_1 < \cdots < I^{\alpha_1}_{\lambda_1} \\ \cdots \\ I^{\alpha_t}_1 < \cdots < I^{\alpha_t}_{\lambda_t}\\ I^{\alpha_{i}}_{j_{i}} \in S_{k}}} \\
		\hspace{10mm}A^{I^{\alpha_1}_1 I^{\alpha_1}_2 \ldots I^{\alpha_t}_{\lambda_t}}_{i_1 \ldots i_{\lambda_0} } \Delta^{I^{\alpha_1}_1 \ldots I^{\alpha_1}_{\lambda_1} \ldots I^{\alpha_t}_1 \ldots I^{\alpha_t}_{\lambda_t}}_{j^{\alpha_1}_1 \ldots j^{\alpha_1}_{\lambda_1} \ldots  j^{\alpha_t}_1 \ldots j^{\alpha_t}_{\lambda_t} }(\nu_1,\ldots,\nu_l\ ; z^{\alpha_1}, \ldots, z^{\alpha_t}) =0\\
		\hspace{65mm} (1 \leq \nu_1 < \cdots < \nu_l \leq n).
		\end{array}
	\end{flalign}
	\renewcommand{\arraystretch}{1}
\end{prop}

\begin{proof}
	 Let $\iota \colon S \hookrightarrow J^{k}(n,m)$ be an $n$-dimensional integral manifold of $\calI$ such that $x_{1},\ldots,x_{n}$ are coordinates of $S$. 
	Then $S$ is an integral manifold of $\calC^{k}$, and we have $p^{\alpha}_{I}(x_{1},\ldots,x_{n})=z^{\alpha}_{I}(x_{1},\ldots,x_{n})\ (1 \leq \alpha \leq m,\ I \in \Sigma_{k})$ by Remark~\ref{rem:canonical-system}.
	Therefore 
\begin{flalign*}
	\iota^{\ast}\Psi &= \sum^m_{t=1} \sum^{l}_{\lambda_{0}=0} \sum_{\lambda_1+ \cdots + \lambda_t=l-\lambda_{0}} \sum_{\alpha_1 < \cdots < \alpha_t}  \sum_{i_1 < \cdots < i_{\lambda_0}} \sum_{\substack{I^{\alpha_1}_1 < \cdots < I^{\alpha_1}_{\lambda_1} \\ \cdots \\ I^{\alpha_t}_1 < \cdots < I^{\alpha_t}_{\lambda_t}\\ I^{\alpha_{i}}_{j_{i}} \in S_{k}}} \notag\\
	& \hspace{5mm}A^{I^{\alpha_1}_1 I^{\alpha_1}_2 \ldots I^{\alpha_t}_{\lambda_t}}_{i_1 \ldots i_{\lambda_0} } dx_{i_1} \wedge \cdots \wedge dx_{i_{\lambda_0}} \wedge dz^{\alpha_1}_{I^{\alpha_1}_1} \wedge dz^{\alpha_1}_{I^{\alpha_1}_2} \wedge \cdots \wedge dz^{\alpha_t}_{I^{\alpha_t}_{\lambda_t}}\hphantom{aaaaaaaaaaaaaaaaaaaaaaaa}
	\end{flalign*}
	\begin{flalign}
	\hphantom{\iota^{\ast}\Psi}& =\sum^m_{t=1} \sum^{l}_{\lambda_{0}=0} \sum_{\lambda_1+ \cdots + \lambda_t=l-\lambda_{0}} \sum_{\alpha_1 < \cdots < \alpha_t} \sum_{i_1 < \cdots < i_{\lambda_0}}  \sum_{\substack{I^{\alpha_1}_1 < \cdots < I^{\alpha_1}_{\lambda_1} \\ \cdots \\ I^{\alpha_t}_1 < \cdots < I^{\alpha_t}_{\lambda_t}\\ I^{\alpha_{i}}_{j_{i}} \in S_{k}}} \sum_{j^{\alpha_{1}}_{1},j^{\alpha_{1}}_{2},\ldots,j^{\alpha_{t}}_{\lambda_{t}}} \notag \\
	&\begin{array}{l}
		\hspace{20mm}A^{I^{\alpha_1}_1 I^{\alpha_1}_2 \ldots I^{\alpha_t}_{\lambda_t}}_{i_1 \ldots i_{\lambda_0} } z^{\alpha_{1}}_{I^{\alpha_{1}}_{1}j^{\alpha_{1}}_{1}}  z^{\alpha_{1}}_{I^{\alpha_{1}}_{2}j^{\alpha_{1}}_{2}}\cdots  z^{\alpha_{t}}_{I^{\alpha_{t}}_{\lambda_{t}}j^{\alpha_{t}}_{\lambda_{t}}} \\
		\hspace{30mm}dx_{i_1} \wedge \cdots \wedge dx_{i_{\lambda_0}} \wedge dx_{j^{\alpha_{1}}_{1}}\wedge dx_{j^{\alpha_{1}}_{2}}\wedge \cdots \wedge dx_{j^{\alpha_{t}}_{\lambda_{t}}} 
		\end{array} \notag \\
	&=\sum^m_{t=1} \sum^{l}_{\lambda_{0}=0} \sum_{\lambda_1+ \cdots + \lambda_t=l-\lambda_{0}} \sum_{\alpha_1 < \cdots < \alpha_t}\sum_{i_1 < \cdots < i_{\lambda_0}}  \sum_{\substack{I^{\alpha_1}_1 < \cdots < I^{\alpha_1}_{\lambda_1} \\ \cdots \\ I^{\alpha_t}_1 < \cdots < I^{\alpha_t}_{\lambda_t}\\ I^{\alpha_{i}}_{j_{i}} \in S_{k}}} \sum_{j^{\alpha_{1}}_{1}<j^{\alpha_{1}}_{2}<\cdots<\j^{\alpha_{t}}_{\lambda_{t}}} \notag \\
	& \begin{array}{l}
		\hspace{20mm}A^{I^{\alpha_1}_1 I^{\alpha_1}_2 \ldots I^{\alpha_t}_{\lambda_t}}_{i_1 \ldots i_{\lambda_0} } 
	H^{I^{\alpha_1}_1 \ldots I^{\alpha_1}_{\lambda_1} \ldots I^{\alpha_t}_1 \ldots I^{\alpha_t}_{\lambda_t}}_{j^{\alpha_1}_1 \ldots j^{\alpha_1}_{\lambda_1} \ldots  j^{\alpha_t}_1 \ldots j^{\alpha_t}_{\lambda_t} }(z^{\alpha_{1}},\ldots,z^{\alpha_{t}})\\
	\hspace{30mm}dx_{i_1} \wedge \cdots \wedge dx_{i_{\lambda_0}} \wedge dx_{j^{\alpha_{1}}_{1}}\wedge dx_{j^{\alpha_{1}}_{2}}\wedge \cdots \wedge dx_{j^{\alpha_{t}}_{\lambda_{t}}}
	\end{array} \notag \\
	& \label{calc2} \begin{aligned}
		&=\sum_{\nu_{1}<\cdots<\nu_{l}}\sum^m_{t=1} \sum^{l}_{\lambda_{0}=0} \sum_{\lambda_1+ \cdots + \lambda_t=l-\lambda_{0}} \sum_{\alpha_1 < \cdots < \alpha_t} \sum_{\substack{i_1 < \cdots < i_{\lambda_0} \\ j^{\alpha_{1}}_{1}<j^{\alpha_{1}}_{2}<\cdots<\j^{\alpha_{t}}_{\lambda_{t}} \\ \{ i_1, \ldots , j^{\alpha_t}_{\lambda_t} \} = \{ \nu_1 ,\ldots, \nu_l \}}} \sum_{\substack{I^{\alpha_1}_1 < \cdots < I^{\alpha_1}_{\lambda_1} \\ \cdots \\ I^{\alpha_t}_1 < \cdots < I^{\alpha_t}_{\lambda_t}\\ I^{\alpha_{i}}_{j_{i}} \in S_{k}}} \\
		&A^{I^{\alpha_1}_1 I^{\alpha_1}_2 \ldots I^{\alpha_t}_{\lambda_t}} _{i_1 \ldots i_{\lambda_0} }
	\Delta^{I^{\alpha_1}_1 \ldots I^{\alpha_1}_{\lambda_1} \ldots I^{\alpha_t}_1 \ldots I^{\alpha_t}_{\lambda_t}}_{j^{\alpha_1}_1 \ldots j^{\alpha_1}_{\lambda_1} \ldots  j^{\alpha_t}_1 \ldots j^{\alpha_t}_{\lambda_t} }(\nu_{1},\ldots,\nu_{l};z^{\alpha_{1}},\ldots,z^{\alpha_{t}}) dx_{\nu_{1}}\wedge \cdots \wedge dx_{\nu_{l}}.
	\end{aligned}
\end{flalign}

Furthermore, $dx_{\nu_{1}}\wedge \cdots \wedge dx_{\nu_{l}}\ (1 \leq \nu_{1} < \cdots < \nu_{l} \leq n)$ are linearly independent on $S$ and hence, by using $\iota^{\ast}\Psi=0$, we obtain
\begin{flalign*}
	\renewcommand{\arraystretch}{2}
	\begin{array}{l}
			\displaystyle
			\sum^m_{t=1} \sum^{l}_{\lambda_{0}=0}\sum_{\lambda_1+ \cdots + \lambda_t=l-\lambda_{0}}  \sum_{\alpha_1 < \cdots < \alpha_t} \sum_{\substack{i_1 < \cdots < i_{\lambda_0} \\ j^{\alpha_1}_1 <j^{\alpha_1}_2< \cdots < j^{\alpha_t}_{\lambda_t}  \\  \{ i_1, \ldots , j^{\alpha_t}_{\lambda_t} \} = \{ \nu_1 ,\ldots, \nu_l \}}}\sum_{\substack{I^{\alpha_1}_1 < \cdots < I^{\alpha_1}_{\lambda_1} \\ \cdots \\ I^{\alpha_t}_1 < \cdots < I^{\alpha_t}_{\lambda_t}\\ I^{\alpha_{i}}_{j_{i}} \in S_{k}}} \\
		\hspace{15mm}A^{I^{\alpha_1}_1 I^{\alpha_1}_2 \ldots I^{\alpha_t}_{\lambda_t}}_{i_1 \ldots i_{\lambda_0} } \Delta^{I^{\alpha_1}_1 \ldots I^{\alpha_1}_{\lambda_1} \ldots I^{\alpha_t}_1 \ldots I^{\alpha_t}_{\lambda_t}}_{j^{\alpha_1}_1 \ldots j^{\alpha_1}_{\lambda_1} \ldots  j^{\alpha_t}_1 \ldots j^{\alpha_t}_{\lambda_t} }(\nu_1,\ldots,\nu_l\ ; z^{\alpha_{1}},\ldots,z^{\alpha_{t}}) =0\\
		\hspace{70mm} (1 \leq \nu_1 < \cdots < \nu_l \leq n). 
	\end{array}
	\renewcommand{\arraystretch}{1}
\end{flalign*}
The converse is proved by substituting a solution of the GMAE for $z$ in (\ref{calc2}).
\end{proof}

\begin{prop}[the case of $r=1$]\label{GMAE-GMAS2}
	For an arbitrary GMAE in the case $r=1$, namely
	\renewcommand{\arraystretch}{2.5}
	\begin{flalign}
		\begin{array}{l}\label{GMAE-r=1}
			\displaystyle
			\sum^m_{t=1} \sum^{l}_{\lambda_{0}=0}\sum_{\lambda_1+ \cdots + \lambda_t=l-\lambda_{0}}  \sum_{\alpha_1 < \cdots < \alpha_t} \sum_{\substack{i_1 < \cdots < i_{\lambda_0} \\ j^{\alpha_1}_1 <j^{\alpha_1}_2< \cdots < j^{\alpha_t}_{\lambda_t}  \\  \{ i_1, \ldots , j^{\alpha_t}_{\lambda_t} \} = \{ \nu_1 ,\ldots, \nu_l \}}}\sum_{\substack{I^{\alpha_1}_1 < \cdots < I^{\alpha_1}_{\lambda_1} \\ \cdots \\ I^{\alpha_t}_1 < \cdots < I^{\alpha_t}_{\lambda_t}\\ I^{\alpha_{i}}_{j_{i}} \in S_{k}}} \\
		\hspace{10mm}A^{I^{\alpha_1}_1 I^{\alpha_1}_2 \ldots I^{\alpha_t}_{\lambda_t}}_{i_1 \ldots i_{\lambda_0} } \Delta^{I^{\alpha_1}_1 \ldots I^{\alpha_1}_{\lambda_1} \ldots I^{\alpha_t}_1 \ldots I^{\alpha_t}_{\lambda_t}}_{j^{\alpha_1}_1 \ldots j^{\alpha_1}_{\lambda_1} \ldots  j^{\alpha_t}_1 \ldots j^{\alpha_t}_{\lambda_t} }(\nu_1,\ldots,\nu_l\ ; z^{\alpha_1}, \ldots, z^{\alpha_t}) =0\\
		\hspace{65mm} (1 \leq \nu_1 < \cdots < \nu_l \leq n), 
		\end{array}
			\renewcommand{\arraystretch}{1}
	\end{flalign}
	a function $z^{\alpha}(x_{1},\ldots,x_{n})\ (1 \leq \alpha \leq m)$ is a solution of (\ref{GMAE-r=1}) if and only if the submanifold $S$ of $J^{k}(n,m)$ given by
	\[
		S=\{(x_i, z^{\alpha}(x_1,\ldots,x_n),z^{\alpha}_I(x_1,\ldots,x_n))\ |\ 1\leq i \leq n,\ 1 \leq \alpha \leq m,\ I \in \Sigma_{k}\}
	\] 
	is an integral manifold of the GMAS generated by the canonical system $\calC^{k}$ on $J^{k}(n,m)$ and the following $l$-form $\Psi$ on $J^{k}(n,m)$:
	\begin{flalign*}
		&\Psi \equiv \sum^m_{t=0} \sum^{l}_{\lambda_{0}=0} \sum_{\lambda_1+ \cdots + \lambda_t=l-\lambda_{0}} \sum_{\alpha_1 < \cdots < \alpha_t}\sum_{i_1 < \cdots < i_{\lambda_0}}  \sum_{\substack{I^{\alpha_1}_1 < \cdots < I^{\alpha_1}_{\lambda_1} \\ \cdots \\ I^{\alpha_t}_1 < \cdots < I^{\alpha_t}_{\lambda_t} \\ I^{\alpha_{i}}_{j_{i}} \in S_{k}}}\\
		&\hspace{15mm}A^{I^{\alpha_1}_1 I^{\alpha_1}_2 \ldots I^{\alpha_t}_{\lambda_t}}_{i_1 \ldots i_{\lambda_0}} dx_{i_1} \wedge \cdots \wedge dx_{i_{\lambda_0}} \wedge dp^{\alpha_1}_{I^{\alpha_1}_1} \wedge dp^{\alpha_1}_{I^{\alpha_1}_2} \wedge \cdots \wedge dp^{\alpha_t}_{I^{\alpha_t}_{\lambda_t}}\mod \calC^{k},
	\end{flalign*}
	where
		$(x_i,z^{\alpha},p^{\alpha}_{I})\ (1 \leq i \leq n,\ 1 \leq \alpha \leq m,\ I \in \Sigma_{k})$ be a canonical coordinate system of $J^{k}(n,m)$.
	\end{prop}

\begin{proof}
	The proof is similar to that of Proposition~\ref{GMAS-GMAE2}.
\end{proof}

\begin{rem}
Note that the above equation (\ref{GMAS}) can be written as
\begin{flalign*}
	&\sum^m_{t=0} \sum_{\lambda_0+\lambda_1+ \cdots + \lambda_t=l}  \sum_{\alpha_1 < \cdots < \alpha_t} \sum_{\substack{i_1 < \cdots < i_{\lambda_0} \\ j^{\alpha_1}_1 < \cdots < j^{\alpha_1}_{\lambda_1} \\ \cdots \\ j^{\alpha_t}_1 < \cdots < j^{\alpha_t}_{\lambda_t}  \\  \{ i_1, \ldots , j^{\alpha_t}_{\lambda_t} \} = \{ \nu_1 ,\ldots, \nu_l \}}}\sum_{\substack{I^{\alpha_1}_1 < \cdots < I^{\alpha_1}_{\lambda_1} \\ \cdots \\ I^{\alpha_t}_1 < \cdots < I^{\alpha_t}_{\lambda_t}}} \\
	&\hspace{10mm}A^{ I^{\alpha_1}_1 I^{\alpha_1}_2 \ldots I^{\alpha_t}_{\lambda_t}}_{i_1 \ldots i_{\lambda_0}}
	\sgn \left(
	\begin{array}{c}
		i_1,\ldots, i_{\lambda_0}, j^{\alpha_1}_1, j^{\alpha_1}_2, \ldots , j^{\alpha_t}_{\lambda_{t}} \\
		\nu_1 ,\ldots, \nu_{\lambda_0},\nu_{\lambda_0+1},\nu_{\lambda_0+2} ,\ldots, \nu_{l} 
	\end{array}
	\right)\\
	&\hspace{20mm}H^{I^{\alpha_1}_1 \ldots I^{\alpha_1}_{\lambda_1}}_{j^{\alpha_1}_1 \ldots j^{\alpha_1}_{\lambda_1} }(z^{\alpha_1}) \cdots H^{I^{\alpha_t}_1 \ldots I^{\alpha_t}_{\lambda_t}}_{j^{\alpha_t}_1 \ldots j^{\alpha_t}_{\lambda_t} }(z^{\alpha_t})=0\hspace{4mm}(1 \leq \nu_1 < \cdots < \nu_l \leq n). 
\end{flalign*}
\end{rem}

Finally, we prove Theorem~\ref{thm:corresponding} in the general case.

\begin{proof}[Proof of Theorem~\ref{thm:corresponding}]
	Let $\calI=\{\calC^{k},\Psi_{\mu}\ |\ 1 \leq \mu \leq r\}_{\mathrm{diff}}$ be a GMAS on $J^{k}(n,m)$. The first part of Theorem~\ref{thm:corresponding} is proved by applying Proposition~\ref{GMAS-GMAE2} for each $\Psi_{\mu}\ (1 \leq \mu \leq r)$. Conversely, for an arbitrary GMAE, the second part of Theorem~\ref{thm:corresponding} is proved by applying Proposition~\ref{GMAE-GMAS2}, which completes the proof.
\end{proof}

\begin{rem}
	From the view point of geometry of jet spaces with differential forms $\Psi_{1},\ldots,\Psi_{r}$ and independence condition $dx_{1}\wedge \cdots  \wedge dx_{n}$,  Theorem~\ref{thm:corresponding} give a local necessary and sufficient condition for a submanifold in jet spaces to be an integral manifold of a GMAS.
\end{rem}

\section{Some examples of GMAS and GMAE}

In this section, we introduce some examples of GMAEs and GMASs. 
We show that the most general nonlinear third order equation which is completely exceptional (\cite{donate-valenti}) and the KdV equation are examples of GMAEs, moreover, the Cauchy--Riemann equations (which are a system of two partial differential equations) are also an example of GMAEs.
\begin{ex}[the case of $n=2,m=1,k=1,r=1,l=2$]
	Let $(x_{1},x_{2},z,p_{1},p_{2})$ be a canonical coordinate system of $J^{1}(2,1)$. In general, the GMAS $\calI$ in this case is generated by the canonical system $\calC^{1}$ on $J^{1}(2,1)$ and 
	\begin{flalign*}
		\Psi &\equiv \sum^2_{\lambda=0}\sum_{i_1 < \ldots <i_\lambda} \sum_{\substack{I_1 < \ldots < I_{2-\lambda} \\ I_j \in S_1}}\\
		&\hspace{20mm}A^{I_1 \ldots I_{2 - \lambda}}_{i_1 \ldots i_\lambda } dx_{i_1}\wedge \cdots \wedge dx_{i_\lambda} \wedge dp_{I_1} \wedge \cdots \wedge dp_{I_{2-\lambda}} \mod \calC^{1}\\
		&\equiv A_{12}dx_{1}\wedge dx_{2} + A^{1}_{1}dx_{1}\wedge dp_{1} + A^{2}_{1}dx_{1}\wedge dp_{2}\\
		&\hspace{20mm}+ A^{1}_{2}dx_{2}\wedge dp_{1} + A^{2}_{2}dx_{2} \wedge dp_{2} + A^{12}dp_{1}\wedge dp_{2} \mod \calC^{1},\\
	\end{flalign*}
	where $A_{12},A^{12}$ and $A^{i}_{j}\ (i,j \in \{1,2\})$ are functions of $x_{1}, x_{2}, z, p_{1}$ and $p_{2}$. By Proposition~\ref{GMAS-GMAE}, the GMAE corresponding to the GMAS $\calI$ is the following form:
	\begin{flalign*}
		&\sum_{\lambda=0}^{2} \sum_{\substack{i_1 < \cdots < i_\lambda \\ j_1 < \cdots <j_{2-\lambda}  \\ \{i_1,\ldots,j_{2-\lambda}\} = \{\nu_1,\nu_2\}}} \sum_{\substack{I_1 < \cdots < I_{2-\lambda} \\ I_j \in S_1}} A^{I_1 \ldots I_{2-\lambda}}_{i_1 \ldots i_\lambda} \Delta_{j_1 \ldots j_{2-\lambda}}^{I_1 \ldots I_{2-\lambda}} (\nu_1 ,\nu_2\ ; z) =0\\
		&\hspace{90mm} (1 \leq \nu_1 < \nu_2 \leq 2),
	\end{flalign*}
	that is,
	\begin{flalign*}
		&A_{12}+A^{1}_{1}z_{x_{1}x_{2}}+A^{2}_{1}z_{x_{2}x_{2}} -A^{1}_{2}z_{x_{1}x_{1}}-A^{2}_{2}z_{x_{2}x_{1}}\\
		&\hspace{54.5mm}+A^{12}(z_{x_{1}x_{1}}z_{x_{2}x_{2}}-z_{x_{2}x_{1}}z_{x_{1}x_{2}})=0.
	\end{flalign*}
	Since $z_{x_{1}x_{2}}=z_{x_{2}x_{1}}$ by putting
	\[
		A:=-A^{1}_{2},\ B:=(A^{1}_{1}-A^{2}_{2})/2,\ C:=A^{2}_{1},\ D:=A_{12},\ E:=A_{12},
	\]
	then we have
	\begin{flalign}\label{sec4:ex-MAE}
		Az_{x_{1}x_{1}} +2Bz_{x_{1}x_{2}} +Cz_{x_{2}x_{2}} + D + E(z_{x_{1}x_{1}}z_{x_{2}x_{2}}-z^{2}_{x_{1}x_{2}})=0.
	\end{flalign}
		The equation (\ref{sec4:ex-MAE}) is the classical Monge--Amp\`ere equation. Therefore, the GMAS in this case corresponds to the classical Monge--Amp\`ere equation.
	\end{ex}

\begin{ex}[the case of $n=2,m=1,k=1,r=1,l=1$]
		Let $(x_{1},x_{2},z,p_{1},p_{2})$ be a canonical coordinate system of $J^{1}(2,1)$. In general, the GMAS $\calI$ in this case is generated by the canonical system $\calC^{1}$ on $J^{1}(2,1)$ and 
		\begin{flalign*}
		\Psi \equiv A^{1}dp_{1} + A^{2}dp_{2} + A_{1}dx_{1} + A_{2}dx_{2} \mod \calC^{1},\\
	\end{flalign*}
where $A^{1},A^{2},A_{1}$ and $A_{2}$ are functions of $x_{1},x_{2},z,p_{1}$ and $p_{2}$.
By Proposition~\ref{GMAS-GMAE}, the GMAE corresponding to the GMAS $\calI$ is the following equation introduced by Example~\ref{ex:overdetermined}:
\begin{flalign}\label{GMAE1}
	\left\{
		\begin{array}{l}
			A^{1}z_{x_{1}x_{1}}+A^{2}z_{x_{1}x_{2}}+A_{1}=0\\
			A^{1}z_{x_{1}x_{2}}+A^{2}z_{x_{2}x_{2}}+A_{2}=0\\
		\end{array}
	\right. .
\end{flalign}
	The equation (\ref{GMAE1}) is the simplest GMAE which is not the classical Monge--Amp\`ere equation.
\end{ex}

\begin{ex}[the case of $n=2,m=1,k=2,r=1,l=2$]
Let 
$(x_{1},x_{2},z,p_{1},p_{2},$ $p_{11},p_{12},p_{22})$
be a canonical coordinate system of $J^{2}(2,1)$. In general, the GMAS $\calI$ in this case is generated by the canonical system $\calC^{2}$ on $J^{2}(2,1)$ and 	\begin{flalign*}
		\Psi \equiv Adx_{1} \wedge dx_{2} &+B_{1}dx_{2}\wedge dp_{11} +B_{2} dx_{2} \wedge dp_{12} + B_{3}dx_{2} \wedge dp_{22} \\
		&+B_{4}dx_{1} \wedge dp_{11} +B_{5} dx_{1} \wedge dp_{12} + B_{6}dx_{1} \wedge dp_{22} \\
		& + C_{1}dp_{11}\wedge dp_{12} + C_{2}dp_{11} \wedge dp_{22} + C_{3} dp_{12} \wedge dp_{22} \mod \calC^{2},
	\end{flalign*}
where $A,B_{i}$ and $C_{j}\ (i \in \{1,2,\ldots,6\},\ j\in \{1,2,3\})$ are functions of $x_{1},x_{2},z,p_{1},$
$p_{2},p_{11},p_{12}$ and $p_{22}$.
By Proposition~\ref{GMAS-GMAE}, the GMAE corresponding to the GMAS $\calI$ is the following form:
\begin{flalign}\label{GMAE2}
		\renewcommand{\arraystretch}{2}
		\begin{array}{l}
		A-B_{1}z_{x_{1}x_{1}x_{1}}-B_{2}z_{x_{1}x_{2}x_{1}}-B_{3}z_{x_{2}x_{2}x_{1}}\\
		\renewcommand{\arraystretch}{1}
		\hspace{5mm}+B_{4}z_{x_{1}x_{1}x_{2}}+B_{5}z_{x_{1}x_{2}x_{2}}+B_{6}z_{x_{2}x_{2}x_{2}}
		+C_{1}
		\left|
		\begin{array}{cc}
			z_{x_{1}x_{1}x_{1}} & z_{x_{1}x_{2}x_{1}} \\
			z_{x_{1}x_{1}x_{2}} & z_{x_{1}x_{2}x_{2}} \\
		\end{array}
		\right|\\
		\renewcommand{\arraystretch}{1}
		\hspace{20mm}+C_{2}\left|
		\begin{array}{cc}
			z_{x_{1}x_{1}x_{1}} & z_{x_{2}x_{2}x_{1}} \\
			z_{x_{1}x_{1}x_{2}} & z_{x_{2}x_{2}x_{2}} \\
		\end{array}
		\right|
		\renewcommand{\arraystretch}{1}
		+C_{3}\left|
		\begin{array}{cc}
			z_{x_{1}x_{2}x_{1}} & z_{x_{2}x_{2}x_{1}} \\
			z_{x_{1}x_{2}x_{2}} & z_{x_{2}x_{2}x_{2}} \\
		\end{array}
		\right| =0.
		\end{array}
\end{flalign}
This equation (\ref{GMAE2}) is known as the most general nonlinear third order equation which is completely exceptional (\cite{donate-valenti}).
In addition, this equation (\ref{GMAE2}) include the Korteweg-de Vries (KdV) equation. Actually, if we put $A=z_{x_{2}}+zz_{x_{1}}, B_{1}=-1$ and other functions equal to 0, then we get the following equation:
	\begin{flalign}\label{KdV}
		z_{x_{2}}+zz_{x_{1}}+z_{x_{1}x_{1}x_{1}}=0.
	\end{flalign}
	The above equation (\ref{KdV}) is the Korteweg-de Vries equation.
\end{ex}

\begin{ex}[the case of $n=2,m=2,k=0,r=2,l_{1}=2,l_{2}=2$]
Recall that $J^{0}(2,2)$ is simply a $(2+2)$ dimensional manifold equipped with the trivial EDS $\calC^{0}=\{0\}$. Let $(x_{1},x_{2},z^{1},z^{2})$ be a local coordinate system of $J^{0}(2,2)$. Then we consider the GMAS $\calI$ generated by the following 2-forms:
	\begin{flalign*}
		\Psi_{1} &:=A dx_{1}\wedge dx_{2} + A_{1} dx_{2} \wedge dz^{1} + A_{2} dx_{2} \wedge dz^{2}\\
		& \hspace{46mm}+ A_{3}dx_{1}\wedge dz^{1} +A_{4} dx_{1} \wedge dz^{2} + A_{5} dz^{1} \wedge dz^{2},\\
		\Psi_{2} &:=B dx_{1}\wedge dx_{2} + B_{1} dx_{2} \wedge dz^{1} + B_{2} dx_{2} \wedge dz^{2}\\
		&\hspace{46mm}+ B_{3}dx_{1}\wedge dz^{1} +B_{4} dx_{1} \wedge dz^{2} + B_{5} dz^{1} \wedge dz^{2},\
			\end{flalign*}
	where $A,B,A_{i}$ and $B_{i}\ (i \in \{1,2,\ldots,5\})$ are functions of $x_{1},x_{2},z^{1}$ and $z^{2}$. Then the GMAE corresponding to the GMAS $\calI$ is in the following:
	\begin{flalign}\label{GMAE3}
		\left\{
			\renewcommand{\arraystretch}{2.4}
			\begin{array}{l}
				A - A_{1}z^{1}_{x_{1}} - A_{2}z^{2}_{x_{1}} + A_{3}z^{1}_{x_{2}} + A_{4}z^{2}_{x_{2}} 
				+ A_{5} \left|
				\renewcommand{\arraystretch}{1}
					\begin{array}{cc}
						z^{1}_{x_{1}} & z^{2}_{x_{1}} \\
						z^{1}_{x_{2}} & z^{2}_{x_{2}}
					\end{array}
					\renewcommand{\arraystretch}{2}
				\right|=0 \\
				B - B_{1}z^{1}_{x_{1}} - B_{2}z^{2}_{x_{1}} + B_{3}z^{1}_{x_{2}} + B_{4}z^{2}_{x_{2}} 
				+ B_{5} \left|
				\renewcommand{\arraystretch}{1}
					\begin{array}{cc}
						z^{1}_{x_{1}} & z^{2}_{x_{1}} \\
						z^{1}_{x_{2}} & z^{2}_{x_{2}}
					\end{array}
				\right|=0 \\
			\end{array}
		\right. .
	\end{flalign}
	This equation (\ref{GMAE3}) include the Cauchy--Riemann equations. In fact, we put $A_{1}=-1, A_{4}=-1, B_{2}=-1, B_{3}=1$ and other functions equal to 0. Then we have the following equations:
	\begin{flalign}\label{cauchy-riemann}
		\left\{
			\begin{array}{l}
				z^{1}_{x_{1}}-z^{2}_{x_{2}}=0 \\
				z^{1}_{x_{2}}+z^{2}_{x_{1}}=0\\
			\end{array}
		\right. .
	\end{flalign}
	These equations (\ref{cauchy-riemann}) are the Cauchy--Riemann equations.
\end{ex}

\section*{Acknowledgement}
The authors would like to thank Akira Kubo, Takayuki Okuda, Yuichiro Taketomi, Hiroshi Tamaru, and Sorin Sabau for useful comments and suggestions.


\begin{thebibliography}{99}
\bibitem{LRC}
V.~V. Lychagin, V.~N. Rubtsov, I.~V. Chekalov, A classification of
  {M}onge-{A}mp\`ere equations, Ann. Sci. \'{E}cole Norm. Sup. (4) 26~(3)
  (1993) 281--308.

\bibitem{Morimoto95}
T.~Morimoto, Monge-{A}mp\`ere equations viewed from contact geometry, in:
  Symplectic singularities and geometry of gauge fields ({W}arsaw, 1995),
  Vol.~39 of Banach Center Publ., Polish Acad. Sci. Inst. Math., Warsaw, 1997,
  pp. 105--121.

\bibitem{Boillat}
G.~Boillat, Sur l'\'{e}quation g\'{e}n\'{e}rale de {M}onge-{A}mp\`ere \`a
  plusieurs variables, C. R. Acad. Sci. Paris S\'{e}r. I Math. 313~(11) (1991)
  805--808.

\bibitem{donate-valenti}
A.~Donato, G.~Valenti, Exceptionality condition and linearization procedure for
  a third order nonlinear {PDE}, J. Math. Anal. Appl. 186~(2) (1994) 375--382.

\bibitem{Cartan}
E.~Cartan, Les syst\`emes de {P}faff, \`a cinq variables et les \'{e}quations
  aux d\'{e}riv\'{e}es partielles du second ordre, Ann. Sci. \'{E}cole Norm.
  Sup. (3) 27 (1910) 109--192.

\bibitem{Hermann}
R.~Hermann, E. {C}artan's geometric theory of partial differential equations,
  Advances in Math. 1~(fasc., fasc. 3) (1965) 265--317.

\bibitem{bryant}
R.~Bryant, P.~Griffiths, L.~Hsu, Hyperbolic exterior differential systems and
  their conservation laws. {I}, Selecta Math. (N.S.) 1~(1) (1995) 21--112.

\bibitem{bryant2}
R.~Bryant, P.~Griffiths, L.~Hsu, Hyperbolic exterior differential systems and
  their conservation laws. {II}, Selecta Math. (N.S.) 1~(2) (1995) 265--323.

\bibitem{Ishikawa-Morimoto}
G.~Ishikawa, T.~Morimoto, Solution surfaces of {M}onge-{A}mp\`ere equations,
  Differential Geom. Appl. 14~(2) (2001) 113--124.

\bibitem{BCG3}
R.~L. Bryant, S.~S. Chern, R.~B. Gardner, H.~L. Goldschmidt, P.~A. Griffiths,
  Exterior differential systems, Vol.~18 of Mathematical Sciences Research
  Institute Publications, Springer-Verlag, New York, 1991.

\bibitem{Cfb}
T.~A. Ivey, J.~M. Landsberg, Cartan for beginners: differential geometry via
  moving frames and exterior differential systems, Vol.~61 of Graduate Studies
  in Mathematics, American Mathematical Society, Providence, RI, 2003.

\bibitem{Yamaguchi82}
K.~Yamaguchi, Contact geometry of higher order, Japan. J. Math. (N.S.) 8~(1)
  (1982) 109--176.

\bibitem{Yamaguchi83}
K.~Yamaguchi, Geometrization of jet bundles, Hokkaido Math. J. 12~(1) (1983)
  27--40.
\end{thebibliography}
\end{document}